\documentclass[12pt,a4paper]{amsart}
\usepackage{amssymb}
\usepackage{mathrsfs}
\usepackage{amsfonts}
\usepackage[active]{srcltx}
\usepackage{enumerate}
\usepackage{color}

\usepackage{amsmath,amssymb,xspace,amsthm}

\newtheorem{theorem}{Theorem}[section]
\newtheorem{remark}[theorem]{Remark}%[section]
\newtheorem{lemma}[theorem]{Lemma}%[section]
\newtheorem{proposition}[theorem]{Proposition}%[section]
\newtheorem{definition}{Definition}%[section]
\numberwithin{equation}{section}

\def\span{\operatorname{span}}

\newcommand{\End}{\operatorname{End}}
\newcommand{\ot}{\otimes}
\newcommand{\ad}{\operatorname{ad}\xspace}

\newcommand{\p}{\partial}

\newcommand{\C}{\ensuremath{\mathbb C}\xspace}
\newcommand{\Q}{\ensuremath{\mathbb{Q}}\xspace}
\renewcommand{\a}{\ensuremath{\alpha}}

\renewcommand{\l}{\ensuremath{\lambda}}

\newcommand{\Z}{\ensuremath{\mathbb{Z}}\xspace}
\newcommand{\N}{\ensuremath{\mathbb{N}}\xspace}

\renewcommand{\leq}{\leqslant}
\renewcommand{\geq}{\geqslant}

\def\LL{\mathcal{L}}
\def\GG{\mathcal{G}}
\def\BB{\mathcal{B}}
\def\I{\mathcal{I}}
\def\J{\mathcal{J}}

\def\sl{\mathfrak{sl}}
\def\gl{\mathfrak{gl}}

\newcommand{\bm}{{\bf m}}
\newcommand{\bn}{{\bf n}}
\newcommand{\bt}{{\bf t}}
\newcommand{\br}{{\bf r}}
\newcommand{\bs}{{\bf s}}
\newcommand{\bi}{{\bf i}}
\newcommand{\bj}{{\bf j}}
\def\detm#1#2{\left|\hskip-3pt\begin{array}{c} #1  \\  #2  \end{array} \hskip -3pt\right|}

\begin{document}

\title[Jet modules for the  centerless Virasoro-like algebra]{Jet modules for the  centerless Virasoro-like algebra}
\footnote{Corresponding author: Genqiang Liu; email:
liugenqiang@amss.ac.cn}
%\footnote{ Corresponding author: Xiangqian Guo; email: guoxq@zzu.edu.cn}
\author{Xiangqian Guo and Genqiang Liu}
\date{} \maketitle

\begin{abstract}
In this paper, we studied the jet modules for the centerless
Virasoro-like algebra which is the Lie algebra of 
the Lie group of the area-preserving diffeomorphisms of a $2$-torus. The jet modules are certain natural
modules over the Lie algebra of semi-direct product of the
centerless Virasoro-like algebra and the Laurent polynomial algebra
in two variables. We reduce the irreducible jet modules to the
finite-dimensional irreducible modules over some
infinite-dimensional Lie algebra and then characterize the
irreducible jet modules with irreducible finite dimensional modules
over $\mathfrak{sl}_2$. To determine the indecomposable jet modules,
we use the technique of polynomial modules in the sense of \cite{BB,
BZ}. Consequently, indecomposable jet modules are described using
modules over the algebra $\BB_+$, which is the ``positive part" of a
Block type algebra studied first by \cite{DZ} and recently by
\cite{IM, I}).
\end{abstract}

 \vskip 5pt \noindent {\em 2000 Math. Subj. Class.:}
17B10, 17B20, 17B65, 17B66, 17B68

\vskip 10pt

\section{Introduction}\label{introduction}
In the study of the representation theory of infinite-dimensional
Lie algebras, one of the most important objects is to classify all
irreducible Harish-Chandra modules, that is, weight modules with
finite-dimensional weight spaces, for some important algebras. There
are some known such classification results for several
infinite-dimensional Lie algebras such as (generalized) Virasoro
algebras (\cite{M, S1, LZ2, BF1, Ma, GLZ1}), (generalized)
Heisenberg-Virasoro algebra (\cite{LZ1, GLiZ, LG}), loop-Virasoro
algebra (\cite{GLZ2}), Witt algebra (\cite{BF}), a class of Lie
algebras of Weyl type (\cite{S2}) and some infinite-dimensional Lie
algebras of Block type (\cite{S3, S4, CGZ, SXX}).

In particular, to obtain such classification for the Witt algebras
Billig and Futorny (\cite{BF}) established a completely new method,
where they used the notions of polynomial modules and jet modules
for the Witt algebras. The motivation of this paper is to generalize
their method to the centerless Virasoro-like algebra. As one of the
steps, we consider the jet modules for this algebra. The centerless
Virasoro-like algebra is  also a kind of $W_\infty$ algebras in
Physics(\cite{Se}). It is the Lie algebra of $\text{SDiff}(T^2)$,
the Lie group of the area-preserving diffeomorphisms of a $2$-torus.
It turns out that the generators of the Virasoro algebra (with or
without central charge) can be constructed as a linear combination
of infinitely many $\text{SDiff}(T^2)$ generators, see \cite{ADFL}.

Let $\Z$ be the group of all integers and $\C$ the field of all
complexes. Let $A=\C[t_1^{\pm1},t_2^{\pm1}]$ be the Laurent polynomial algebra
in two variables $t_1$ and $t_2$. The centerless \textbf{Virasoro-like algebra} $\GG$ is the Lie
subalgebra of the rank-$2$ Witt algebra (the derivation algebra of $A$) spanned by
the operators with divergence zero as follows:
$$d_{\bm}=t_1^{\bm(1)}t_2^{\bm(2)}(\bm(2)t_1\frac{\partial}{\partial t_1}-\bm(1)t_2\frac{\partial}{\partial t_2}),\ \bm\in \Z^2\setminus\{0\}.$$

It is straight to check that
\begin{equation}\label{def}\left.\begin{split}
[d_\bm,d_\bn] & =\det{\bn\choose\bm}d_{\bm+\bn},\\
\end{split}\right.\quad\forall\ \bm,\bn\in \Z^2\setminus\{0\}, i=1,2,\end{equation}
where $\bm=(\bm(1), \bm(2))$, $\bn=(\bn(1), \bn(2))$ and
$\det{\bn\choose\bm}=\bn(1)\bm(2)-\bn(2)\bm(1)$.

%The centerless Virasoro-like algebra $\GG$ is one of the
%infinite-dimensional simple Lie algebra with simplest structure;
%indeed, maybe only the centerless Virasoro algebra (rank-$1$ Witt
%algebra) has simpler structure. As a subalgebra of the rank-$2$ Witt
%algebra, ${\GG}$ is also known as the \textbf{skew derivation Lie
%algebra}. It  is also a special type of Block algebras studied in
%\cite{Bl, DZ}. 
Note that $\GG$ has no weight zero part, i.e., no
Cartan subalgebra. So it does not even have an intrinsic way to
define weight modules. To get around this, several authors  look at
$\Z^2$-graded and $\Z$-graded modules over $\GG$.
% The Virasoro-like algebra
%$\GG$ contains many standard Heisenberg Lie subalgebras and has
%features similar to the Virasoro algebra and Heisenberg algebra.
%Some relations between the Virasoro-like algebra $\GG$ and the
%generalized Clifford algebras were given in \cite{K}.

%Clearly  $\GG$ is $\Z^2$-graded and to reflect this fact we add degree derivation
%$\p_1=t_1\frac{\partial}{\partial t_1}, \p_2=t_2\frac{\partial}{\partial t_2}$.
% Let $\overline{\GG}=\GG\oplus \C \p_1\oplus \C \p_2$ and extend the Lie bracket as
 %$$[\p_i, d_{\bm}]=\bm(i)d_{\bm},\  [\p_1,\p_2]=0, \ \ i=1,2.$$

Many irreducible $\Z^2$-graded and $\Z$-graded modules over $\GG$
have been constructed. For example, a large class of uniformly
bounded $\Z^2$-graded ${\GG}$-module were constructed in
[LT1], a large class of $\Z^2$-graded irreducible generalized
highest weight ${\GG}$-module were constructed in
\cite{BZ}, and a class of $\Z$-graded irreducible highest weight
$\GG$-modules were constructed in \cite{BGLZ} (these modules can be
regarded as certain kind of Whittaker modules in \cite{GLi}). In
\cite{WZ}, the $\Z^2$-graded Verma modules over $\GG$ relative to
some total order were investigated. In \cite{LT2}, the structure of
irreducible $\Z^2$-graded modules over the universal central extension $\widetilde{\GG}$ (which is called the Virasoro-like algebra) of $\GG$ with nonzero central actions
and all homogeneous spaces finite dimensional were determined.
Recently, the results of Whittaker modules for $\widetilde{\GG}$ in \cite{GLi} were
generalized to the algebra of semi-direct product of $\GG$ and the
rank-$2$ Laurent polynomials in \cite{TWX}.

%The derivation Lie algebra of the centerless Virasoro-like algebra
%¡¥L and the automorphism group of this derivation Lie algebra were
%studied in [13], while the structure of automorphism group of ¡¥L
%was considered in [14].

In \cite{LS}, the irreducible $\Z^2$-graded and $\Z$-graded
$\GG$-modules were divided into two classes: generalized highest
weight modules and uniformly bounded modules. However, they did not
give a description of the structure of irreducible uniformly bounded
modules. We want to settle this problem for the $\Z^2$-graded case
using the technique developed by Billig and Futorny in \cite{BF}. To
this aim, we first studied the jet modules in this paper. Another
reason to research jet modules over $\GG$ is that the  natural
modules coming from the 2-torus, and the tensor field modules are
jet modules. And it is reasonable to guess that any irreducible
uniformly bounded $\Z^2$-graded $\GG$-module is isomorphic some
subquotient of tensor field modules.

%% Combining with the result of \cite{LS}, we can
%give a complete classification of irreducible $\Z^2$-graded
%$\GG$-modules with finite dimensional homogeneous spaces.

Our paper is organized as follows. In Section \ref{jet modules}, we
first introduce the notation of jet modules for the Virasoro-like
algebra $\GG$ and reduce the study of the jet modules to the study
of finite dimensional modules over a Lie algebra $\LL$. In Section
\ref{irreducible jet modules}, we provided an (equivalent) weaker
condition for irreducible $\GG$-modules to be jet modules. Then
irreducible $\LL$-modules are classified in Section
\ref{LL-modules}, using $\mathfrak{sl}_2$-modules. This give rise to
a completely description of irreducible jet modules via finite
dimensional irreducible $\sl_2$-modules. The last section is devoted
to the indecomposable jet modules, which are shown to be polynomial
modules. More precisely, we show that the indecomposable jet modules
are in one-to-one correspondence with the finite dimensional
indecomposable module over a Lie algebra $\BB_+$, which is the
``positive part" of a special Block type Lie algebra, studied first
by Djokovic and Zhao (\cite{DZ}) and recently by Iohara and Mathieu
(\cite{IM}, \cite{I}). We mention that using the results of Section
5, one can also describe the irreducible jet modules as in Section
4, however, our methods in Section 4 are more straightforward and
does not involve the concept of polynomial modules as in \cite{B1}.

When we were submitting this paper, we learned that there are also
several papers which deal with similar problems. For example, in \cite{JL}, 
the authors classified the irreducible $\GG$-modules with the condition that
the action of $A$ being associative. In \cite{BT}, the authors considered the
same kind of modules for the algebras of divergence zero vector fields on
$n$-dimensional tori, which is just the Virasoro-like algebra considered in this paper when $n=2$.
However, they also require that the action of $A_n=\C[t_1^{\pm1},\cdots,t_n^{\pm1}]$ are associative.
The main differences of the present paper with the above mentioned papers is
that we only assume that the action of $A$ is nonzero; and this is just the main 
difficulty of our methods inducing associativity from being nonzero. 
In fact, such kind of modules are of particular importance when we try to classify all 
irreducible weight Virasoro-like modules with finite-dimensional weight spaces, 
as similar situation occurs in the classification of such modules for the Witt algebras (see \cite{BF}).

\section{Jet modules for the Virasoro-like algebra $\GG$}\label{jet modules}

Let $\mathbb{Z}_+$ and $\N$ be the sets of all nonnegative integers
and all positive integers respectively.
%For convenience, we also make the following conventions in this
%paper. We use $i,j,...$ to denote elements in $\Z$, use $a,b,...$ to
%denote elements in $\C$, use element $m,n,r,s...$ to denote elements
%in $\Z^d$, and use $\a,\b,...$ to denote elements in $\C^d$ for some
%$d\in\N$.
All vector spaces and algebras are over $\C$. %For a matrix $A$, we denote its $(i,j)$-entry by $A(i,j)$.
For a Lie algebra $\LL$, we denote its enveloping algebra by
$U(\LL)$. For any $\bm\in\Z^2$, we denote by $\bm(1),\bm(2)$ its
entries, i.e., $\bm=(\bm(1),\bm(2))$; let $|\bm|=\bm(1)+\bm(2)$. For
any $\br\in \Z_+^2$ and $\bm\in\Z^2$ we denote
$\bm^{\br}=(\bm(1))^{\br(1)}(\bm(2))^{\br(2)}$ and
$\br!=(\br(1))!(\br(2))!$. We also have the generalized binomial
coefficients
${\bn\choose\bm}={\bn(1)\choose\bm(1)}{\bn(2)\choose\bm(2)}$ for
$\bm, \bn\in\Z_+^2$ with $\bm(1)\leq \bn(1)$ and $\bm(2)\leq
\bn(2)$.

In \cite{B1}, Y. Billig introduced the concept of jet modules for
the Witt algebra $W_n$, that is, the Lie algebra of vector fields on
an $n$-dimensional torus. The jet modules originate from the modules
of the jet bundles of tensor fields on the torus. Similarly, we can
define the jet modules for the Virasoro-like algebra $\GG$.

 For convenience denote
$\bt^{\bn}=t_1^{\bn(1)}t_2^{\bn(2)}$ for any $\bn\in\Z^2$. Then $A$
has a natural $\GG$-module structure:
\begin{equation}\label{act_GA}\left.\begin{split}
d_\bm \bt^{\bn} & =\det{\bn\choose\bm}\bt^{\bm+\bn},\\
\end{split}\right.
\quad\quad\forall\ \bm, \bn\in \Z^2\setminus\{0\},
i=1,2.\end{equation} We we can form the semi-direct Lie algebra
${\GG}\ltimes A$ whose Lie bracket is
\begin{equation}\label{def_GA}
[x+f, y+g]=[x,y]+xg-yf,\quad \forall\ x,y\in\GG, f,g\in A.
\end{equation}
It is clear that $A$ is an abelian Lie subalgebra of ${\GG}\ltimes A$.

\begin{definition} For any $\l\in \C\setminus \{0\}$, a $\GG$-module $V$ is called a \textbf{jet module} of parameter $\l$ if
\begin{enumerate}
\item[$(J1)$] $V$ is a $\Z^2$-graded ${\GG}\ltimes A$-module with bounded dimensions of homogenous  spaces,
i.e.,
\begin{equation}\label{grading}\begin{split}
& V=\bigoplus_{\bm\in\Z^2} V_{\bm},\ \ \ \dim V_{\bm}\leq N\ \rm{for\ some\ }N\in\N,\\
\end{split}\end{equation}
and the $\GG$-module structure is just the restriction; %of the $\GG\ltimes A$-module structure;

\item[$(J2)$] The action of all $t^{\bm} ( \bm\in\Z^2\setminus\{0\})$ are bijective and quasi-associative, i.e.,
\begin{equation}\label{asso}
t^{\bm}t^{\bn}v=\l\ t^{\bm+\bn}v, \ \forall\ \bm,\bn, \in\Z^2\setminus\{0\}\  with\  \bm+\bn\neq 0, v\in V.
\end{equation}
\end{enumerate}
Denote the category consisting of all jet $\GG$-modules by
$\mathcal{J}$. A jet module is called \textbf{irreducible} (resp.
\textbf{indecomposable}) if it is an irreducible (resp.
indecomposable) $({\GG}\ltimes A)$-module.
\end{definition}

A homomorphism $f$ between two jet modules $V=\bigoplus_{\bm\in\Z^2}
V_{\bm}, W=\bigoplus_{\bm\in\Z^2} W_{\bm}$ is  an isomorphism if
there is a bijection $\varphi$ of $\Z^2$ such that
$f(V_{\bm})=W_{\varphi(\bm)}$ for any $\bm\in \Z^2$.

\begin{remark} Note that an irreducible jet $\GG$-module and an irreducible $\GG$-module are not the same thing in general.
Indeed, there exist jet modules which are irreducible as
$({\GG}\ltimes A)$-modules but decomposable as $\GG$-modules, for
example, $A$ is an irreducible $({\GG}\ltimes A)$-module via
the natural way but $A=A'\oplus\C t^0$ as a
$\GG$-module, where $A'=\span\{t^{\bm}\ |\
\bm\in\Z^2\setminus\{0\}\}$.
\end{remark}

%\begin{remark}\label{cond 1} Among the first condition, if all the central elements act as scalars,
%then the homogeneous spaces having bounded dimension implies those
%scalars being $0$. So if the module is irreducible over $\GG\ltimes
%A$, then central actions are automatically trivial (see Corollary
%4.8 of \cite{LS} or consider $V$ as modules over the Heisenberg
%subalgebra of $\GG$ and use a result of \cite{F} to deduce the
%triviality of central action).
%\end{remark}

\begin{remark}\label{free}
Our concept of jet modules are slightly different from that of
Billig in \cite{B1}. Billig's jet modules for $W_n$ require that all
$t^{\bm}$ acts associatively, i.e., $\l=1$ in our case. However, all
properties within our interest can be deduced if we replace
``associativity" with ``quasi-associativity" (See \cite{GLiuZ}).
Indeed, if we take $(t^{\bm})'=t^{\bm}/\l$, then we have
$$(t^{\bm})'(t^{\bn})'v=(t^{\bm+\bn})'v, \ \forall\
\bm,\bn\in\Z^2, v\in V,$$ which justify the terminology
quasi-associative. We also note that the bijectivity of the actions
of $t^{\bm}, 0 \neq\bm\in\Z^2$ on a jet module $V$ indicates $\dim
V_{\bn}=\dim V_{\bm+\bn}$ for all $\bm,\bn\in\C^2$.
\end{remark}

\begin{remark}\label{cond 2}
Note that ${\GG}\ltimes A=({\GG}\ltimes A')\oplus \C t^0$ is a
direct sum of ideals , so the action of $t^0$ is a scalar and does
not affect the module structure for an irreducible $({\GG}\ltimes
A)$-module. As we will see in Section \ref{irreducible jet modules},
when the module $V$ is irreducible over ${\GG}\ltimes A$, the second
condition can be weakened as
\begin{enumerate}
\item[$(J2')$] The action of $A'$ is nonzero.
\end{enumerate}
\end{remark}

Let $V=\bigoplus_{\bm\in \Z^2} V_{\bm}$ be a jet $\GG$-module. Thanks to Remark \ref{free}, we can
choose linear transformations $L(\bm): V_0\rightarrow V_0$ such that
\begin{equation}\label{d_m by L(m)}
d_{\bm}v=t^{\bm}L(\bm)v,\ \forall\ \bm\in \Z^2\setminus\{0\}, v\in
V_0.
\end{equation}
Using the Lie bracket \eqref{def}, \eqref{act_GA} and
\eqref{def_GA}, we can deduce that
\begin{equation}\label{d_mt^n by L(m)}
d_{\bm}(t^{\bn}v)=t^{\bm+\bn}L(\bm)v+\det{\bn\choose\bm}t^{\bm+\bn}v,\
\forall\ \bm,\bn\in \Z^2\setminus\{0\}, v\in V_0
\end{equation}
and the commutator
\begin{equation}\label{def_LL}
[L(\bm),L(\bn)]=\det{\bn\choose\bm}\big(L(\bm+\bn)-L(\bm)-L(\bn)\big),\
\forall\ \bm,\bn\in \Z^2\setminus\{0\}.
\end{equation}
Denote by $\LL$ the Lie algebra defined by the Lie bracket in
\eqref{def_LL} and by abuse of language we may assume that $L(\bm),
\bm\in\Z^2$ forms a basis $\LL$, that is,
$\LL=\bigoplus_{\bm\in\Z^2\setminus\{0\}}\C L(\bm)$. Thus $V_0$ can
be viewed as a finite dimensional module over $\LL$. In Section
\ref{LL-modules}, we will  describe the irreducible finite
dimensional modules over $\LL$. This allow us to characterize the
irreducible jet modules using $\mathfrak{sl}_2$-modules. In Section
\ref{indecomposable jet modules}, we will show that $L(\bm)$, as a
linear transformation on $V_0$, is a polynomial in $\bm\in\Z^2$.
Hence $V_0$ is a polynomial $\LL$-module  in the sense of \cite{BB,
BZ}. Recall that a finite dimensional $\LL$-module $W$ is called a
\textbf{polynomial module} if there is a basis $\{v_1,\dots ,v_k\}$
of $W$ such that
$$L(\bm)v_i=\sum_{j=1}^k f_j(\bm)v_j,\ \bm=(\bm(1),\bm(2))\in\Z^2,\ 1\leq i\leq k, $$
where all $f_j(\bm)$ are polynomials in $\bm(1),\bm(2)$.
% with coefficients in $\End(W)$.

\section{Irreducible jet modules}\label{irreducible jet modules}

In this section and the next, we will deduce the structure of
irreducible jet modules. Before this, we first weaken the second
condition for an irreducible jet module, which is of interest in its
own right. First note that ${\GG}\ltimes A=({\GG}\ltimes A')\otimes \C
t^0$, where $A'=\sum_{\bm\in\Z^2\setminus\{0\}}\C t^{\bm}.$ Hence,
for any irreducible  $({\GG}\ltimes A)$-module, the action of $t^0$ is
a scalar and its value does not affect the module structure. Thus we
can change the value of $t^0$ appropriately due to our convenience.

\begin{theorem}\label{weak condition} Let $V$ be an irreducible $({\GG}\ltimes
A)$-module satisfying the conditions $(J1)$ and $(J2')$, then $V$ is
a jet module of parameter $\l$, for some nonzero scalar $\l\in \C$.
\end{theorem}

%Fix a $\GG$-module $V$ with gradation \eqref{grading} satisfying conditions $(J1)$ and $(J2')$.
To prove Theorem \ref{weak condition}, we need the following
technique lemmas. Their proofs are easy and we leave them to the
reader. For a proof of the first lemma please see Lemma 3.1 of
\cite{GLiuZ}. The proof of the second one follows easily from the
fact that $[U(A),[U(A),U(\GG\ltimes A)]]=0$, similar to the proof of
Lemma 3.2 of \cite{GLiuZ}.

\begin{lemma}\label{A-mod} Let
$V=\bigoplus_{\bn\in\Z^2}V_{\bn}$ be an irreducible $\Z^2$-graded
module with bounded dimensions of homogeneous spaces over the
associative algebra $A$, then $\dim V_{\bn}\leq 1$ for all
$\bn\in\Z^2$.
\end{lemma}
%
%\begin{proof} Let $\Gamma$ be the additive subgroup of $\Z^2$
%generated by all $\bm$ with $t^{\bm}$ being injective on $V$.
%
%For $\bn\in \Z^d$, let $U_{\bn}$ be the $\bn$-th homogeneous
%subspace of $U(A)$ with respect to the $\Z^2$-gradation on $A$. Then
%$U_0$ is a commutative associative algebra, and each nonzero
%homogenous subspace $V_{\bn}$ is a finite-dimensional $U_0$-module.
%Without loss of generality we may assume that $V_0\neq 0$. Take any
%common eigenvector $v\in V_0$ of $U_0$ and let $W=U(A)v$. It is
%clear that $W_{\bn}=U_{\bn}v$.
%
%
%\end{proof}

\begin{lemma}\label{nilp} Suppose that $V$ is an irreducible $({\GG}\ltimes A)$-module
and $x\in U(A)$. If $xv=0$ for some nonzero $v\in V$, then $x$ is
locally nilpotent on $V$.
\end{lemma}

\begin{proposition}\label{inj_or_nil} Let $V$ be an irreducible $({\GG}\ltimes A)$-module
satisfying condition $(J1)$. Then either all $t^{\bm}, \bm\neq 0$
act injectively on $V$ or all $t^{\bm}, \bm\neq 0$ act locally
nilpotently on $V$.
% or all $t^{\bm}$ with $m\neq 0$ act locally nilpotently on $V$ .
%\begin{itemize}
%\item[(1)] All $t^{\bm}$ acts injectively on $V$;
%\item[(2)] All $t^{\bm}$ acts locally nilpotently on $V$;.
%\end{itemize}
\end{proposition}
\begin{proof} Suppose that $V=\bigoplus_{\bn\in \Z^2} V_{\bn}$
 and  $t^{\bm}$ acts injectively on $V$ for some
$\bm\in\Z^2\setminus\{0\}$.

\noindent\textbf{Claim 1.} For any $\bn\in\Z^2$ with
$\det{\bn\choose\bm}\neq 0$, $t^{\bn}$ acts nilpotently on $V$ if it
acts locally nilpotently on $V$.

Take any $\bn\in\Z^2$ such that $\det{\bn\choose \bm}\neq0$ and
$t^{\bn}$ acts locally nilpotently on $V$. Then $t^{-\bn}t^{\bn}$ is
locally nilpotent on $V$ and hence nilpotent on each homogeneous
space of $V$. Since the dimensions of all homogeneous spaces have an
upper bound, there exists $N\in\N$ such that
$\big(t^{-\bn}t^{\bn}\big)^{N}V=0$, i.e.,
$\big(t^{-\bn}\big)^{N}\big(t^{\bn}\big)^{N}V=0$. Here we note that
$(t^{\bn})^k$ denotes $t^{\bn}\cdots t^{\bn}$ rather than $t^{k\bn}$
for all $k\in\N$ and one should be careful with this notion
throughout this section.

Now suppose that there exist positive integers $k$ and $N_k $ with
$1\leq k\leq N$ such that
$\big(t^{-\bn}\big)^{k}\big(t^{\bn}\big)^{N_k}V=0$. Then we have
\begin{equation}\begin{split}
0& = d_{\bm+\bn}\big(t^{-\bn}\big)^{k}\big(t^{\bn}\big)^{N_k}V\\
 & = N_k\det{\bn\choose \bm}t^{\bm+2\bn}\big(t^{-\bn}\big)^{k}\big(t^{\bn}\big)^{N_k-1}V
    -k\det{\bn\choose \bm}t^{\bm}\big(t^{-\bn}\big)^{k-1}\big(t^{\bn}\big)^{N_k}V.
\end{split}\end{equation}
Applying $t^{\bn}$ to the above equation, we deduce that
$t^{\bm}\big(t^{-\bn}\big)^{k-1}\big(t^{\bn}\big)^{N_k+1}V=0$ and
hence $\big(t^{-\bn}\big)^{k-1}\big(t^{\bn}\big)^{N_k+1}V=0$ since
$t^{\bm}$ is injective. By downward induction on $k$, we obtain that
$\big(t^{\bn}\big)^{N_0}V=0$ for some $N_0\in\N$. Claim 1 is proved.

\noindent\textbf{Claim 2.} $t^{q\bm/p}$ acts injectively on $V$ for
any $p,q\in\Z$ whenever $q\bm/p\in\Z^2\setminus\{0\}$.

Assume $q/p< 0$ without loss of generality. Suppose on the contrary
that the action of $t^{q\bm/p}$ on $V$ is not injective, then
$t^{q\bm/p}$ acts on $V$ locally nilpotently by Lemma \ref{nilp}.
Consequently, $(t^{\bm})^{-q}(t^{q\bm/p})^p$ acts nilpotently on
each homogeneous space and hence also on the whole space $V$. There
exists an $N'\in\N$ such that
$\big((t^{\bm})^{-q}(t^{q\bm/p})^p\big)^{N'}V=0$. Since $t^{\bm}$ is
injective, we have $(t^{q\bm/p})^{pN'}V=0$. Choose a minimal
$N''\in\N$ such that $(t^{q\bm/p})^{N''}V=0$.

Take any $\bn\in\Z^2$ such that $\det{\bn\choose\bm}\neq 0$, then we
have
$$0=d_{\bn-q\bm/p}(t^{q\bm/p})^{N''}V=N''\det{q\bm/p\choose
\bn}t^{\bn}(t^{q\bm/p})^{N''-1}V=0.$$ Since
$(t^{q\bm/p})^{N''-1}V\neq0$ by the choice of $N''$, we see that
$t^{\bn}$ is locally nilpotent on $V$ by Lemma \ref{nilp}. By Claim
1, $t^{\bn}$ is nilpotent on $V$ and we can choose a minimal
$N_0\in\N$ such that $(t^{\bn})^{N_0}V=0$. Then
$$0=d_{\bm-\bn}(t^{\bn})^{N_0}V=N_0\det{\bn\choose
\bm}t^{\bm}(t^{\bn})^{N_0-1}V=0,$$ which implies that $t^m$ is not
injective, contradiction. Claim 2 follows.

\noindent\textbf{Claim 3.} $t^{\bn}$ is injective on $V$ for all
$\bn\in\Z^2\setminus\{0\}$.

%By Claim 1, we may assume that $\bm(1)$ and $\bm(2)$ are coprime.
By Claim 2, it suffices to prove this claim for $\bn\in\Z^2$ with
$\det{\bn\choose\bm}\neq0$. Take one such $\bn$ and suppose that
$t^{\bn}$ is locally nilpotent on $V$. Then Claim 1 implies that
$t^{\bn}$ is nilpotent on $V$. Find $k\in\N$ such that
$(t^{\bn})^kV=0$.

Without loss of generality, we also assume that $\dim V_0\geq \dim
V_{\br}$ for all $\br\in\Z^2$. Then we have
$t^{\bm}V_{i\bm}=V_{(i+1)\bm}$ and $\dim V_{i\bm}=\dim V_{0}$ for
all $i\in\Z$. Then for any $v\in V_{-k\bm}$ we have
$$0=(d_{\bm-\bn})^k(t^{\bn})^{k}v=k!\det{\bn\choose
\bm}(t^{\bm})^kv+t^{\bn}u,$$ for some $u\in V_{-\bn}$. This means
that the map $$t^{\bn}: V_{-\bn} \rightarrow V_{0}$$ is surjective.
The fact that $\dim V_{0}\geq \dim V_{-\bn}$ indicates that
$t^{\bn}$ is bijective on $V_{-\bn}$. %In particular, $\dim V_{-\bn}=\dim V_0$.
Replacing $V_0$ with $V_{-i\bn}$ in the above
argument, we may deduce that $t^{\bn}$ is bijective on all
$V_{-i\bn}, i\in\N$. Thus $t^{\bn}$ can not be nilpotent,
contradiction. Claim 3 follows.

Finally, this lemma follows from Claim 3 and Lemma \ref{nilp}.
\end{proof}

\begin{lemma}\label{inj} Let $V$ be an irreducible $(\GG\ltimes A)$-module
satisfying condition $(J1)$.  If $t^{\bm}$ acts injectively on $V$
for any $\bm\in\Z^2\setminus\{0\}$, then there exists
$\l\in\C\setminus\{0\}$ such that $(t^{\bm}t^{\bn}-\l
t^{\bm+\bn})V=0$ for all $\bm,\bn,\bm+\bn\in\Z^2\setminus\{0\}$.
\end{lemma}

\begin{proof} Suppose that $V=\bigoplus_{\bn\in \Z^2} V_{\bn}$. Since each $t^{\bm}$ is injective, we have $\dim
V_{\bm}=\dim V_{\bn}$ for all $\bm,\bn\in \Z^2$. By Lemma
\ref{A-mod}, we see that $V$ contains an irreducible
, $A$-submodule $W=\oplus_{\bm\in Z^2}W_{\bm}$ with $W_{\bm}\subset V_{\bm}$. From remark \ref{cond 2}, we may assume that
$t^0$ acts as a nonzero scalar $\l\in\C$. Then there exist
$\lambda_{\bm,\bn}\in \C\setminus\{0\}$ such that
$$t^{\bm}t^{\bn} v=\lambda_{\bm,\bn}t^{\bm+\bn}v,\ \forall\ \bm, \bn\in \Z^2, v\in W_0.$$
Set
$$T_{\bm,\bn}=t^{\bm}t^{\bn} -\lambda_{\bm,\bn} t^{\bm+\bn},\ \forall\ \bm, \bn\in \Z^2.$$
We can adjust the value of $\l$ so that
$\l_{(1,0),(-1,0)}=\l_{(1,0),(1,0)}$. Note that
$\l_{\bm,\bn}=\l_{\bn,\bm}$ and $\l_{\bm,0}=\l$. By Lemma \ref{nilp}
we have that $T_{\bm,\bn}$ is locally nilpotent on $V$ for all
$\bm,\bn\in \Z^2$.

 Choose a basis
$\{v_1,\cdots,v_k\}$ of $V_0$, then $\{t^{\bm}v_1,\cdots,
t^{\bm}v_k\}$ is a basis of $V_{\bm}$ for all $\bm\in \Z^2$. There
exist $k\times k$ matrices $A_{\bm,\bn}$ such that
$$t^{\bm}t^{\bn}(v_1,\cdots,v_k)=t^{\bm+\bn}(v_1,\cdots,v_k)A_{\bm,\bn}.$$ Since $A$ is commutative, we
see that
$$A_{\bm_1,\bn_1}A_{\bm_2,\bn_2}=A_{\bm_2,\bn_2}A_{\bm_1,\bn_1},\ \forall\ \bm_1,\bm_2,\bn_1,\bn_2\in \Z^2.$$

So there exists a complex matrix $S$ such that
$B_{\bm,\bn}=S^{-1}A_{\bm,\bn}S$ are upper triangular matrices for
all $\bm,\bn\in \Z^2$. Set
$$(w_1,\cdots,w_k)=(v_1,\cdots,v_k)S.$$
We have
$$t^{\bm}t^{\bn}(w_1,\cdots,w_k)=t^{\bm+\bn}(w_1,\cdots,w_k)B_{\bm,\bn}$$
and
\begin{equation}\label{inj_1} T_{\bm,\bn}(w_1,\cdots,w_k)=t^{\bm+\bn}(w_1,\cdots,w_k)(B_{\bm,\bn}-\lambda_{\bm,\bn}I_k),\end{equation}
where $I_k$ is the identity matrix of rank $k$. Since $T_{\bm,\bn}$
is locally nilpotent on $V$, we see that all the diagonal entries of
$B_{\bm,\bn}$ are $\lambda_{\bm,\bn}$. In particular,
\begin{equation}\label{inj_2}
T_{\bm, \bn}w_l\in \sum_{i=0}^{l-1}\C t^{\bm+\bn}w_i,\ \forall\
\bm,\bn\in \Z^2, l=1,\cdots, k\ {\textrm {where}} \ w_0=0,
 \end{equation}
and hence $(T_{\bm,\bn})^{k}w_k=0$.

For any $\bm,\bn\in \Z^2$ and $u\in V_{-\bm-\bn}=\sum_{i=1}^{k}\C
t^{-\bm-\bn}w_i$, we have by (\ref{inj_2}) that
\begin{equation}\label{inj_3}\aligned
T_{\bm,\bn}u\in & \sum_{i=0}^{k}\C T_{\bm,\bn}t^{-\bm-\bn}w_i=
\sum_{i=0}^{k}\C t^{-\bm-\bn}T_{\bm,\bn}w_i\\ & \subseteq
\sum_{i=0}^{k-1}\C t^{-\bm-\bn}t^{\bm+\bn}w_i
 \subseteq\sum_{i=0}^{k-1}\C \big(T
_{\bm+\bn,-\bm-\bn}+\l_{\bm+\bn,-\bm-\bn}t^0\big)w_i\\& \subseteq
\sum_{i=0}^{k-1}\C w_i. \endaligned\end{equation} Set
$V'_0=\sum_{i=0}^{k-1}\C w_i$, then we have
$T_{\bm,\bn}V_{-\bm-\bn}\subset V'_0$ for all $\bm,\bn\in \Z^2$.

%Take any $\bm,\bn\in \Z^2$. For any $\bm_1,\cdots,\bm_l\in
%\Z^2$, let $\bm'=\sum_{i=1}^l\bm_i+l(\bm+\bn)$. Then
%we have
%$$\aligned 0
%&=t^{-\bm'}\big(\prod_{i=1}^ld_{\bm_i}\big)(T_{\bm,\bn})^{l}w_l\equiv l!t^{-\bm'}\Big(\prod_{i=1}^l[d_{\bm_i}, T_{\bm,\bn}]\Big)w_l \ (\hspace {-.38cm} \mod V_0')\\
%&\equiv l!t^{-\bm'} \prod_{i=1}^l \Big(\detm{\bm}{\bm_i}t^{\bm_i+\bm}t^{\bn}+\detm{\bn}{\bm_i}t^{\bm_i+\bn}t^{\bm}-\detm{\bm+\bn}{\bm_i}\l_{\bm,\bn}t^{\bm_i+\bm+\bn}\Big)w_l \  (\hspace {-.38cm} \mod V_0')\\
%&\equiv l!t^{-\bm'} \Big(\prod_{i=1}^l\big(\detm{\bm}{\bm_i}\l_{\bm_i+\bm,\bn}+\detm{\bn}{\bm_i}\l_{\bm_i+\bn,\bm}-\detm{\bm+\bn}{\bm_i}\l_{\bm,\bn}\big)t^{\bm_i+\bm+\bn}\Big)w_l \  (\hspace {-.38cm} \mod V_0')\\
%&\equiv
%l!t^{-\bm'} \prod_{i=1}^l\Big(\detm{\bm}{\bm_i}\l_{\bm_i+\bm,\bn}+\detm{\bn}{\bm_i}\l_{\bm_i+\bn,\bm}-\detm{\bm+\bn}{\bm_i}\l_{\bm,\bn}\Big)
%\prod_{i=1}^l t^{\bm_i+\bm+\bn}w_l \  (\hspace {-.38cm} \mod V_0').
%%&\equiv \lambda'\prod_{i=1}^l\Big({\bm\choose\bm_i}\l_{\bm_i+\bm,\bn}+{\bn\choose\bm_i}\l_{\bm_i+\bn,\bm}-{\bm+\bn\choose\bm_i}\l_{\bm,\bn}\Big) t^0w_l\ (\hspace {-.38cm} \mod V_0')\\
%\endaligned$$ %for some $0\neq \lambda'\in\C$.
%By the injectivity of all $t^{\bm}, t^{\bn}$ and $t^{\bm_i}$,
%replacing $\bm_i$ with some $\bl\in\Z^2$, we obtain that
%\begin{equation}\label{inj_4}
%\detm{\bm}{\bl}\l_{\bl+\bm,\bn}+\detm{\bn}{\bl}\l_{\bl+\bn,\bm}-\detm{\bm+\bn}{\bl}\l_{\bm,\bn}=0,\
%\forall\ \bl,\bm,\bn\in \Z^2.
%\end{equation}

Given any $\bm,\bn, \br\in \Z^2$, set $\br'=\bm+\bn+\br$ and we have
$$\aligned 0
&=t^{-\br'}d_{\br}T_{\bm,\bn}w_k\equiv t^{-\bm-\bn-\br}[d_{\br}, T_{\bm,\bn}]w_k \mod V_0'\\
&\equiv t^{-\br'}\left(\det{\bm\choose \br}t^{\bm+\br}t^{\bn}+\det{\bn\choose\br}t^{\bm}t^{\bn+\br}-\det{\bm+\bn\choose\br}\l_{\bm,\bn}t^{\bm+\bn+\br}\right)w_k \mod V_0'\\
&\equiv t^{-\br'} \left(\det{\bm\choose\br}\l_{\bm+\br,\bn}+\det{\bn\choose\br}\l_{\bm,\bn+\br}-\det{\bm+\bn\choose\br}\l_{\bm,\bn}\right)t^{\bm+\bn+\br}w_k \mod V_0'.\\
\endaligned$$ %for some $0\neq \lambda'\in\C$.
By the injectivity of $t^{\bm}, t^{\bn}$ and $t^{\br}$, we obtain
that
\begin{equation}\label{inj_4}
\det{\bm\choose\br}\l_{\bm+\br,\bn}+\det{\bn\choose\br}\l_{\bm,\bn+\br}-\det{\bm+\bn\choose
\br}\l_{\bm,\bn}=0,\ \forall\ \bm,\bn,\br\in \Z^2.
\end{equation}

Take $\bm=j\bn$ for some $j\in\Z$ and $\br\in\Z^2$ such that
$\det{\bn\choose\br}\neq 0$ in \eqref{inj_4}, we have
\begin{equation}\label{inj_5}j\l_{j\bn+\br,\bn}+\l_{j\bn,\bn+\br}-(j+1)\l_{j\bn,\bn}=0,\
\forall\ \det{\bn\choose\br}\neq 0.\end{equation} Take $j=1$ we get
$\l_{\bn+\br,\bn}=\l_{\bn,\bn}$ for all $\det{\bn\choose\br}\neq 0$.
Then we have $\l_{j\bn+\br,\bn}=\l_{\bn,\bn}$ and
$$\l_{j\bn,\bn+\br}=\l_{\bn+\br,\bn+\br}=\l_{\bn+\br,\bn}=\l_{\bn,\bn}.$$
Substitute these into \eqref{inj_5}, we can deduce that
$\l_{j\bn,\bn}=\l_{\bn,\bn}$ for $j\neq -1$. Consequently,
$\l_{\bm,\bn}=\l_{\bn,\bn}$ for all $\bm,\bn,\bm+\bn\neq 0$.

Taking $\br=-\bm-\bn$ in \eqref{inj_4}, we see that
$\l_{-\bn,\bn}=\l_{-\bm,\bm}$ for all $\det{\bm\choose\bn}\neq 0$
and hence for all $\bm,\bn\neq 0$. Combining these results with the
assumption that $\l_{(1,0),(1,0)}=\l_{(1,0),-(1,0)}$, we obtain that
$\l_{\bm,\bn}=\l_{\bn,\bn}$ for all $\bm,\bn\neq 0$. Note that
\begin{equation*}\begin{split}
(t^{\bm}t^{-\bm})(t^{\bn}t^{-\bn})w_1 & = \l_{\bn, -\bn}t^0(t^{\bm}t^{-\bm})w_1 = \l_{\bm,-\bm}\l_{\bn,-\bn}(t^{0})^2w_1 \\
=(t^{\bm}t^{\bn})(t^{-\bm}t^{-\bn})w_1 & =
\l_{\bm,\bn}\l_{-\bm,-\bn}t^{\bm+\bn}t^{-\bm-\bn}w_1=\l_{\bm,\bn}\l_{-\bm,-\bn}\l_{\bm+\bn,-\bm-\bn}t^0w_1.
\end{split}\end{equation*}
Recalling that $t^0$ acts as the scalar $\l$, we have
$\l_{\bm,\bn}=\l$ for all $\bm,\bn\in\Z^2$, as desired.

Finally, we have
$$w_1\in V'=\{v\in V\ |\ \big(t^{\bm}t^{\bn}-\l t^{\bm+\bn}\big)v=0,\ \forall\  \bm, \bn\in \Z^2\},$$
which is nonzero. Since the space spanned by $t^{\bm}t^{\bn}-\l
t^{\bm+\bn}, \bm,\bn\in \Z^2$, is stable under the action of
$\GG\ltimes A$, so $V'$ is an $(\GG\ltimes A)$-submodule of $V$. By
the irreducibility of $V$, we see that $V'=V$. Our result follows.
\end{proof}

%\begin{remark}\label{remark} From the proof of the above lemma, we see
%that we can adjust the action of $t^0$ so that
%$(t^{\bm}t^{\bn}-t^{\bm+\bn})V=0$ for all $\bm,\bn\in\Z^2$ under
%the assumption of Lemma \ref{inj}.
%\end{remark}

\begin{lemma}\label{nil} Let $V$ be an irreducible $(\GG\ltimes A)$-module
satisfying condition $(J1)$. If $t^{\bm}$ acts locally nilpotently
on $V$ for any $\bm\in\Z^2\setminus\{0\}$, then $t^{\bm}V=0$ for all
$\bm\in\Z^2\setminus\{0\}$.
\end{lemma}

\begin{proof} Without loss of generality, we may assume that $t^0V=0$.
Set $\deg(d_{\bm})=\deg(t^{\bm})=\bm$, then we have a
$\Z^2$-gradation on $\GG\ltimes A$ and $U(A)$. Denote $U=U(A)$ for
short and let $U_{\bm}$ be the homogeneous component of degree $\bm$. It is
clear that any element in $U_0$ is nilpotent on each $V_{\bm}$.
Since $\dim V_{\bm}, \bm\in\Z^2$ have a uniform upper bound, we have
$U_0^NV=0$ for some $N\in\N$. In particular,
$(t^{\bn}t^{-\bn})^NV=0$ for any $\bn\neq 0$.

Choose any $\bm, \bn\in\Z^2$ such that $\det{\bm\choose\bn}\neq 0$,
then for any $v\in V$ we have
$$0=d_{\bm}(t^{\bn}t^{-\bn})^Nv=\det{\bn\choose\bm}Nt^{\bm+\bn}(t^{\bn})^{N-1}(t^{-\bn})^{N}v-\det{\bn\choose\bm}Nt^{\bm-\bn}(t^{-\bn})^{N-1}(t^{\bn})^{N}v.$$
Applying $t^{\bn}$ to the above equation, we deduce
\begin{equation}\label{nil_1}
t^{\bm-\bn}(t^{-\bn})^{N-1}(t^{\bn})^{N+1}V=0,\ \forall\
\det{\bm\choose\bn}\neq 0. \end{equation} First we want to prove the
following equation for all $1\leq j\leq N$ by induction on $j$:
\begin{equation}\label{nil_2}
t^{\bm_1-\bn}\cdots
t^{\bm_j-\bn}(t^{-\bn})^{N-j}(t^{\bn})^{N+j}V=0,
\end{equation}for any $m_1,\cdots,m_j\in\Z^2 $ with
$\sum_{i=1}^j\epsilon_i\bm_i\notin\Q\bn, \epsilon_i\in\{0,1\}$. When
$j=1$ \eqref{nil_2} is just \eqref{nil_1}. Now suppose that
\eqref{nil_2} holds for some $j$ with $1\leq j<N$. Applying
$d_{\bm}$ to \eqref{nil_2} for any $\bm$ with
$\bm+\sum_{i=1}\epsilon_i\bm_i\notin\Q\bn$ for all
$\epsilon_i\in\{0,1\}$, we have
\begin{equation*}\begin{split}
0 = & d_{\bm}t^{\bm_1-\bn}\cdots t^{\bm_j-\bn}(t^{-\bn})^{N-j}(t^{\bn})^{N+j}V\\
  = & (N+j)\det{\bn\choose\bm}t^{\bm+\bn}t^{\bm_1-\bn}\cdots t^{\bm_j-\bn}(t^{-\bn})^{N-j}(t^{\bn})^{N+j-1}V\\
    & -(N-j)\det{\bn\choose\bm}t^{\bm-\bn}t^{\bm_1-\bn}\cdots t^{\bm_j-\bn}(t^{-\bn})^{N-j-1}(t^{\bn})^{N+j}V.\\
\end{split}\end{equation*}
Applying $t^{\bn}$ to this equation, we get
$t^{\bm-\bn}t^{\bm_1-\bn}\cdots
t^{\bm_j-\bn}(t^{-\bn})^{N-j-1}(t^{\bn})^{N+j+1}V=0$, that is,
\eqref{nil_2} holds for all $j$ with $1\leq j\leq N$. Take $j=N$ in
\eqref{nil_2}, we get $t^{\bm_1-\bn}\cdots
t^{\bm_N-\bn}(t^{\bn})^{2N}V=0.$

Next we prove for all $0\leq j\leq 2N$ that
\begin{equation}\label{nil_3}\begin{split}
t^{\bm_1-\bn}\cdots t^{\bm_{N+j}-\bn}(t^{\bn})^{2N-j}V=0,\ \forall\
\sum_{i=1}^j\epsilon_i\bm_i\notin\Q\bn, \epsilon_i\in\{0,1\}.
\end{split}\end{equation}
The equation \eqref{nil_3} for $j=0$ has been proved. Now suppose
that \eqref{nil_3} holds for some $j$ with $0\leq j<2N$. Applying
$d_{\bm-2\bn}$ to \eqref{nil_3} for any $\bm$ with
$\bm+\sum_{i=1}^j\epsilon_i\bm_i\notin\Q\bn$ for all
$\epsilon_i\in\{0,1\}$, we have
\begin{equation*}\begin{split}
0 & =  d_{\bm-2\bn}t^{\bm_1-\bn}\cdots t^{\bm_{N+j}-\bn}(t^{\bn})^{2N-j}V \\
  & =  (2N-j)\det{\bn\choose\bm}t^{\bm-\bn}t^{\bm_1-\bn}\cdots t^{\bm_{N+j}-\bn}(t^{\bn})^{2N-j-1}V,\\
\end{split}\end{equation*}
which is just \eqref{nil_3} for $j+1$. So \eqref{nil_3} holds for
all $0\leq j\leq 2N$. Take $j=2N$ there we obtain that
\begin{equation}\label{nil_4}\begin{split}
t^{\bm_1-\bn}\cdots t^{\bm_{3N}-\bn}V=0,\ \forall\
\sum_{i=1}^{3N}\epsilon_i\bm_i\notin\Q\bn, \epsilon_i\in\{0,1\}.
\end{split}\end{equation}

Now for any $\bn_1,\cdots,\bn_{3N}$, take any
$\bn\in\Z^2\setminus\sum_{\epsilon_i\in\{0,1\}}\C\big(
\sum_{i=1}^{3N}\epsilon_i\bn_i\big),$ and let $\bm_i=\bn_i+\bn$ for
all $i=1,\cdots, 3N$. By \eqref{nil_4}, we have that
$t^{\bn_1}\cdots t^{\bn_{3N}}V=0$, that is $A^{3N}V=0$. Then it is
clear that the subspace
$$V'=\{ v\in V \mid  Av=0 \}$$
is a nonzero $(\GG\ltimes A)$-submodule of $V$. Therefore $V'=V$ and
$AV=0$, as desired.
\end{proof}

Now we can give the proof of the Theorem \ref{weak condition}.
Indeed, let $V$ be an irreducible $(\GG\ltimes A)$-module satisfying
the conditions $(J1)$ and $(J2')$, then by Proposition
\ref{inj_or_nil}, Lemma \ref{inj} and Lemma \ref{nil}, we see that
condition $(J2)$ holds. % all $t^{\bm}, \bm\in\Z^2$ are injective and $(t^{\bm}t^{\bn}-t^{\bm+\bn})V=0$ for all $\bm,\bn\in\Z^2$.
Thus $V$ is a jet module.

\section{Finite dimensional modules over $\LL$}\label{LL-modules}

%From Section \ref{step2}, we see that it suffices to determine the
%finite dimensional irreducible modules over an infinite dimensional
%Lie algebra %$\LL$. The Lie algebra
%$\LL=\bigoplus_{\bm\in\Z^2\setminus\{0\}}L(\bm)$ with Lie bracket
%$$[L(\bm), L(\bn)]=\det{\bn\choose\bm}\Big(L(\bm+\bn)-L(\bm)-L(\bn)\Big).$$

In this section, we will give a description of the irreducible jet
modules over $\GG$. From the argument in the last paragraph of
Section \ref{jet modules}, we see that the homogeneous space of a
jet module admits a natural $\LL$-module structure, where $\LL$ is
the Lie algebra defined by \eqref{def_LL}. To determine this module
structure, we first make some preparations on the ideals of this
Lie algebra $\LL$. %Now fix a finite dimensional irreducible module $V_{0}$ over $\LL$.
For any $k\geq 2$, define the subspace
$$\I_{k}=\left\{\sum_{i=1}^na_iL(\bm_i)\ |\
\sum_{i=1}^na_i\bm_i^{\br}=0,\ \forall\ \br\in\Z_+^2, 2\leq
|\br|\leq k\right\},$$ where $|\br|=\br(1)+\br(2)$,
$\bm^{\br}=\bm(1)^{\br(1)}\bm(2)^{\br(2)}$ and we have make the
convention that $0^0=1$. Then we have the following properties:
\begin{proposition}\label{property} Assume $k,l\in\N$ with $k,l\geq 2$,
then we have
\begin{enumerate}
\item\label{property_(1)} $\I_{k}$ are co-finite dimensional ideals of $\LL$ for all $k\geq
2$;
\item\label{property_(2)} $\I_{k+1}\subseteq \I_{k}$;
\item\label{property_(3)} $\cap_{k=2}^{+\infty}\I_{k}=0$;
\item\label{property_(4)} $[\I_{k},\I_{l}]\subseteq\I_{k+l-1}$;
\item\label{property_(5)} $\LL/\I_{2}\cong\sl_2$;
%\item\label{property_(6)} For any co-finite dimensional ideal $J$ of $\LL$, there exists $k\geq 2$ such that $\I_{k}\subseteq J$;
\item\label{property_(7)} $[\LL,\LL]=\left\{\sum_{i=1}^na_iL(\bm_i)\ |\ \sum_{i=1}^na_i\bm_i=0, a_1,\dots,a_n\in \C\right\}$;
\item\label{property_(8)} $\LL/(\I_2\cap[\LL,\LL])\cong\sl_2\oplus\C z_1\oplus\C
z_2$, where $z_1$ and $z_2$ are central.
\end{enumerate}
\end{proposition}

\begin{proof} (1) Take any $\sum_{i=1}^n a_iL(\bm_i)\in \I_k$. Then for
any $\bn\in\Z^2$, we have $$[L(\bn), \sum_{i=1}^n
a_iL(\bm_i)]=\sum_{i=1}^n
a_i\det{\bm_i\choose\bn}\big(L(\bm_i+\bn)-L(\bm_i)-L(\bn)\big).$$
Then for any $\br\in\Z_+^2, 2\leq |\br|\leq k$, we have
\begin{equation*}\begin{split}
  & \sum_{i=1}^n a_i\det{\bm_i\choose\bn}\big((\bm_i+\bn)^{\br}-(\bm_i)^{\br}-(\bn)^{\br}\big)\\
= & \sum_{i=1}^n a_i \big(\bm_i(2)\bn(1)-\bm_i(1)\bn(2)\big)\sum_{\bs,\br-\bs\in\Z_+^2\setminus\{0\}}{\br\choose\bs}\bm_i^{\bs}\bn^{\br-\bs}\\
= & \sum_{\bs,\br-\bs\in\Z_+^2\setminus\{0\}}{\br\choose\bs}\Big(\bn^{\br-\bs+(1,0)}\sum_{i=1}^n a_i \bm_i^{\bs+(0,1)}-\bn^{\br-\bs+(0,1)}\sum_{i=1}^na_i\bm_i^{\bs+(1,0)}\Big)=0.\\
\end{split}\end{equation*}
Thus $[L(\bn), \sum_{i=1}^n a_iL(\bm_i)]\in\I_k$ and $\I_k$ is an
ideal of $\LL$.

For any element $x=\sum_{i=1}^na_iL(\bm_i)$ we can assign a vector
$v_x=\big(\sum_{i=1}^na_i\bm_i^{\br}\big)_{\br\in\Z_+^2, 2\leq
|\br|\leq k}$. Note that there exist $x_1,\cdots,x_{l}$ such that
$\span\{v_{x_1},\cdots,v_{x_{l}}\}=\span\{v_x\ |\ x\in\LL\}$. Then
it is easy to see that $\LL=\span\{x_1,\cdots,x_{l}\}+\I_k$, that
is, $\I_k$ has finite codimension in $\LL$.

(2) is clear. To prove (3) we take any
$\sum_{i=1}^na_iL(\bm_i)\in\cap_{k=2}^{+\infty}\I_k$, then
$\sum_{i=1}^na_i\bm_i^{\br}=0$ for all $\br\in\Z_+, |\br|\geq 2$.
Since all $\bm_i$ are distinct, we may choose $\br\in\N^2$ such that
all $\bm_i^{\br}$ are distinct. Then we have
$\sum_{i=1}^na_i(\bm_i^{\br})^l=0$ for all $l\in\N$, forcing all
$a_i=0$. Hence $\cap_{k=2}^{+\infty}\I_k=0$.

(4) Take any $x=\sum_{i=1}^pa_iL(\bm_i)\in\I_k$ and
$y=\sum_{j=1}^qb_jL(\bn_j)\in\I_l$, then we have
\begin{equation*}\begin{split}
[\sum_{i=1}^pa_iL(\bm_i),
\sum_{j=1}^qb_jL(\bn_j)]=\sum_{i=1}^p\sum_{j=1}^q
a_ib_j\det{\bn_j\choose\bm_i}\big(L(\bm_i+\bn_j)-L(\bm_i)-L(\bn_j)\big).
\end{split}\end{equation*}
For any $\br\in\Z_+^2$ with $2\leq |\br|\leq k+l-1$, we can compute
\begin{equation}\label{property_1}\begin{split}
 & \sum_{i=1}^p\sum_{j=1}^q a_ib_j\det{\bn_j\choose\bm_i}\big((\bm_i+\bn_j)^{\br}-(\bm_i)^{\br}-(\bn_j)^{\br}\big)\\
=& \sum_{i=1}^p\sum_{j=1}^q a_ib_j\det{\bn_j\choose\bm_i}\sum_{\bs,\br-\bs\in\Z_+^2\setminus\{0\}}{\br\choose\bs}(\bm_i)^{\bs}(\bn_j)^{\br-\bs}\\
%=& \sum_{\bs,\br-\bs\in\Z_+^2\setminus\{0\}}{\br\choose\bs} \sum_{i=1}^p\sum_{j=1}^q a_ib_j\detm{\bn_j}{\bm_i}(\bm_i)^{\bs}(\bn_j)^{\br-\bs}\\
=& \sum_{\bs,\br-\bs\in\Z_+^2\setminus\{0\}}{\br\choose\bs} \sum_{i=1}^p\sum_{j=1}^q a_ib_j\big(\bm_i(2)\bn_j(1)-\bm_i(1)\bn_j(2)\big)(\bm_i)^{\bs}(\bn_j)^{\br-\bs}.\\
\end{split}\end{equation}
Note that $$\sum_{i=1}^p\sum_{j=1}^q
a_ib_j\bm_i(2)\bn_j(1)(\bm_i)^{\bs}(\bn_j)^{\br-\bs}=\Big(\sum_{i=1}^p
a_i\bm_i(2)(\bm_i)^{\bs}\Big)\Big(\sum_{j=1}^qb_j\bn_j(1)(\bn_j)^{\br-\bs}\Big)=0,$$
where the last equation holds because either $2\leq |\bs|+1\leq k$
or $2\leq |\br-\bs|+1\leq l$. Similarly $$\sum_{i=1}^p\sum_{j=1}^q
a_ib_j\bm_i(1)\bn_j(2)(\bm_i)^{\bs}(\bn_j)^{\br-\bs}=0.$$ This
indicates the element in (\ref{property_1}) equals to $0$ and
hence $[x,y]\in\I_{k+j-1}$ as desired. % and $[\I_k,\I_l]\subseteq\I_{k+l-1}$.

(5) Denote by $e,f,h$ the images of $-L(1,0), L(0,1)$ and
$L(1,1)-L(1,0)-L(0,1)$ in $\LL/\I_2$. It is easy to check that
$L(i,j)-ij(L(1,1)-L(1,0)-L(0,1))-i^2L(1,0)-j^2L(0,1)\in\I_2$ for all
$i,j\in\Z_+$, hence $\LL/\I_2=\C e\oplus\C h\oplus\C f$.
%We define the following linear map $$\varphi:\ \LL/\I_2\rightarrow \sl_2,\quad  \sum_{i=1}^pa_iL(\bm_i)\mapsto .$$
From (\ref{def_LL}) and the definition of $\I_2$, we get
\begin{equation}\aligned
& [-L(1,0), L(0,1)]=L(1,1)-L(1,0)-L(0,1),\\\\
& [L(1,1)-L(1,0)-L(0,1),-L(1,0)]\\
=& -\big(L(2,1)-L(1,1)-L(1,0)\big)+\big(L(1,1)-L(0,1)-L(1,0)\big)\\
=&-2L(1,0) -\left(L(2,1)-2L(1,1)-2L(1,0)+L(0,1)\right)\in -2L(1,0)+\I_2,\\\\
 & [L(1,1)-L(1,0)-L(0,1),L(0,1)]\\
=& -\big(L(1,2)-L(1,1)-L(0,1)\big)+\big(L(1,1)-L(0,1)-L(1,0)\big)\\
=&-2L(0,1) -\left(L(1,2)-2L(1,1)-2L(0,1)+L(1,0)\right)\in -2L(0,1)+\I_2 \\
\endaligned\end{equation}
Then $[e,f]=h,\ [h,e]=2e, [h,f]=-2f$. So $\LL/\I_2\cong\sl_2$ as Lie
algebras.

%(6) Let $J$ be a co-finite dimensional ideal $J$ of $\LL$ and
%suppose that $\I_{k}\not\subseteq J$ for all $k\geq 2$. Take any
%$k_0\geq 2$ and $x_1\in \I_{k_0}\setminus J$. Since
%$\cap_{k=2}^{+\infty}\I_{k}=0$, there is some $k'\geq k_0$ such that
%$x_1\not\in \I_{k'}$. We can find $k_1\geq k_0$ such that
%$x_1\in\I_{k_1}\setminus\I_{k_1+1}$. Then take
%$x_2\in\I_{k_1+1}\setminus J$ and similarly we have $x_2\in
%\I_{k_2}\setminus\I_{k_2+1}$ for some $k_2>k_1$. Inductively we
%reach the sequences $(x_1, x_2,\cdots)$ and $(k_1,k_2\cdots)$ such
%that $k_1>k_2>\cdots$ and
%$x_i\in\I_{k_i}\setminus(\I_{k_i+1}+J)\subseteq
%\I_{k_i}\setminus(\I_{k_{i+1}}+J)$. Now we have $x_1,x_2,\cdots$ are
%linearly independent, contracting the fact that $J$ is co-finite
%dimensional in $\LL$. So we must have $\I_k\subseteq J$ for some
%$k\geq 2$.

%(6) can be verified straightforward.
(6) By (\ref{def_LL}), $L(\bm)-\bm(1)L(1,0)-\bm(2)L(0,1)\in [\LL,\LL]$ for any nonzero
$\bm\in \Z^2$. Thus (6) is true.

(7) We note that
$[\LL,\LL]+\I_2=\LL$ and hence $[\LL,\LL]/(\I_2\cap
[\LL,\LL])=\LL/\I_2\cong\sl_2$ and $\I_2/(\I_2\cap
[\LL,\LL])=\LL/[\LL,\LL]$ is $2$-dimensional. More precisely, we
have $\LL/[\LL,\LL]=\C z_1\oplus\C z_2$, where $z_1$ and $z_2$ are
the images of $L(1,0)$ and $L(0,1)$ in $\LL/[\LL,\LL]$ respectively.
As in \eqref{property_(5)} we still denote by $e,f,h$ the images of
$-L(1,0), L(0,1)$ and $L(1,1)-L(1,0)-L(0,1)$ in $\LL/\I_2$
respectively. Thus %$$\LL/(\I_2\cap[\LL,\LL])=[\LL,\LL]/(\I_2\cap [\LL,\LL])\oplus \I_2/(\I_2\cap [\LL,\LL])\cong \sl_2\oplus\C z_1\C z_2.$$
$$\frac{\LL}{\I_2\cap[\LL,\LL]}=\frac{[\LL,\LL]}{\I_2\cap
[\LL,\LL]}\oplus \frac{\I_2}{\I_2\cap [\LL,\LL]}\cong
\frac{\LL}{\I_2}\oplus\frac{\LL}{[\LL,\LL]}\cong\sl_2\oplus\C
z_1\oplus\C z_2,$$ which is a reductive Lie algebra with center $\C
z_1\oplus\C z_2$. The corresponding projection is given by
\begin{equation}\label{projection}\begin{split}
& \pi:\ \LL\rightarrow \sl_2\oplus\C z_1\oplus\C z_2,\ \pi(L(\bm))=
\left(\begin{array}{cc}
\bm(1)\bm(2) & -\bm(1)^2\\
\bm(2)^2 & -\bm(1)\bm(2)\\
\end{array}\right)+\bm(1)z_1+\bm(2)z_2.
\end{split}\end{equation}
\end{proof}

For a Lie algebra $\I$, denote $\I^2=[\I,\I],\I^3=[\I,\I^2],\dots,
\I^{k+1}=[\I,\I^k]$ and $\I^w=\cap_{k=2}^{+\infty}\I^{k}$.

\begin{lemma}\label{l4.2} If $\I$ is a finite dimensional Lie algebra such that  $\I^w=0$,
then $\I$ is a nilpotent Lie algebra.
\end{lemma}
\begin{proof}Since $\I$ is  finite dimensional, there is a positive integer $n$ such that $\I^n=\I^w=0$. So
$\I$ is nilpotent.
\end{proof}
Now we can give the description of finite dimensional irreducible
modules over $\LL$.
\begin{theorem}\label{finite} Let $V$ be a finite dimensional irreducible
$\LL$-modules. Then $\I_2\cap[\LL,\LL]$ acts trivially on $V$, i.e,
$V$ is an irreducible module over $\sl_2\oplus\C z_1\oplus\C z_2$.
\end{theorem}

\begin{proof}
Let $\rho: \LL\rightarrow \gl(V)$ be the representation of $\LL$ in
$V$.  By
Proposition \ref{property} (2)(3)(4), $(\I_2/\ker(\rho))^w=0$. Then from Lemma \ref{l4.2},
$\I_2/\ker(\rho)$ is a nilpotent ideal of $\LL/\ker(\rho)$.
Thus $\I_2/\ker(\rho)$ is in the radical of
$\LL/\ker(\rho)$. By a well known result on finite dimensional Lie
algebra (see Proposition 19.1 of \cite{H}), we know that
$\LL/\ker(\rho)$ is a reductive Lie algebra with trivial or
$1$-dimensional center. So $\dim(\I_{2}/\ker(\rho))\leq 1$ and all
elements in $\I_2$ acts as scalars on $V$. Note that any element in
$\I_2\cap [\LL,\LL]$ has trace $0$ on $V$, so $\I_2\cap
[\LL,\LL]\subseteq \ker(\rho)$.
\end{proof}

Using Theorem \ref{finite} and formula \eqref{d_mt^n by L(m)}, we
can determine the structure of irreducible jet modules explicitly.
Before doing this, we first define some irreducible jet modules. For
any irreducible finite dimensional $\sl_2$-module $U$,
$\a=(\alpha_1,\alpha_2)\in\C^2$ and $\l\in \C$, we can define a
$({\GG}\ltimes A)$-module structure on $\C[t_1^{\pm},
t_2^{\pm}]\otimes U$ by
\begin{equation*}
t^{\bm}(t^{\bn}\otimes v)=\l\ t^{\bm+\bn}\otimes v
\end{equation*} and
\begin{equation*}\begin{split}
    d_{\bm}(t^{\bn}\otimes v )= t^{\bm+\bn}\otimes
    \left(\left(\begin{array}{cc}\bm(1)\bm(2)&-\bm(1)^2\\\bm(2)^2&-\bm(1)\bm(2)\end{array}\right)
    +\det{\textbf{\a}+\bn\choose\bm}\right)v\\
\end{split}\end{equation*}
for any $\bm,\bn\in\Z^2$ and $v\in U$. We denote this module by
$F(\alpha,\l, U)$. One check that $F(\alpha,\l, U)\cong
F(\alpha',\l', U')$ as jet modules if and only if $\a-\a'\in \Z^2,
\l=\l'$ and $U\cong U'$ as $\sl_2$-modules.

\begin{theorem}\label{char irre jet} If $V$ is an irreducible jet $\GG$-module,
then $V \cong F(\alpha,U,\l)$ for some irreducible finite
dimensional $\sl_2$-module $U$ and $\a=(\alpha_1,\alpha_2)\in \C^2$,
$\l\in\C\setminus\{0\}$.
\end{theorem}

\begin{proof}
Let $V=\bigoplus_{\bm\in\Z^2}V_{\bm}$ be an irreducible jet
$\GG$-module. We denote $t^\bm v$ by $t^\bm\otimes v$ for all $v\in V_0, \bm\in \Z^2\setminus\{0\}$. Then there exists nonzero $\l\in\C$ such that
$t^{\bm}(t^{\bn}\otimes v)=\l t^{\bm+\bn}\otimes v$ for all
$\bm,\bn\in\Z^2$ and $v\in V_0$. By the argument at the end of Section
\ref{jet modules}, we see that $V_0$ is a finite dimensional
irreducible module over the infinite dimensional Lie algebra $\LL$
defined by \eqref{def_LL}. By Theorem \ref{finite}, $V_0$ can be
viewed as a finite dimensional irreducible module over
$\sl_2\oplus\C z_1\oplus\C z_2$, as given in Proposition
\ref{property} \eqref{property_(8)}.

We can give the action of $\LL$ on $V_0$ explicitly. There exist
$\a_1,\a_2\in\C$ such that $z_1$ and $z_2$ act as scalars $-\a_2$
and $\a_1$ respectively. Using the projection given by
\eqref{projection}, we can write out the $\LL$-module action
explicitly as
$$L(\bm)v=\left(\begin{array}{cc}
\bm(1)\bm(2) & -\bm(1)^2\\
\bm(2)^2 & -\bm(1)\bm(2)\\
\end{array}\right)v+(\bm(2)\a_1-\bm(1)\a_2)v.$$

Since $t^{\bn}$ acts bijectively on $V$, we can identify
$V_{\bn}=t^\bn V_0$.  Then by \eqref{d_m by
L(m)} and \eqref{d_mt^n by L(m)} we can deduce that
\begin{equation*}\begin{split}
d_{\bm}v = &  t^{\bm}\otimes L(\bm)v\\
         = &  \left(\begin{array}{cc}
\bm(1)\bm(2) & -\bm(1)^2\\
\bm(2)^2 & -\bm(1)\bm(2)\\
\end{array}\right)t^{\bm}\ot v+\detm{\bf{\a}}{\bm} t^{\bm}\ot v,
\end{split}\end{equation*} where $\bf\a=(\a_1,\a_2)$ and
\begin{equation*}\begin{split}
d_{\bm}(t^{\bn}\ot v)= & d_{\bm}t^{\bn}v=t^{\bn}\ot d_{\bm}v+\detm{\bn}{\bm}t^{\bm+\bn}\ot v\\
                     = & \left(\begin{array}{cc}
\bm(1)\bm(2) & -\bm(1)^2\\
\bm(2)^2 & -\bm(1)\bm(2)\\
\end{array}\right) t^{\bm+\bn}\ot v+\detm{\bf{\a}+\bn}{\bm}t^{\bm+\bn}\ot v,
\end{split}\end{equation*}
for all $\bm, \bn\in\Z^2$ and $v\in V_0$. So $A\ot V_0$ is submodule
of $V$. By the irreducibility of $V$, we have $V=A\ot V_0$. Then $V
\cong F(\alpha,U,\l)$.
\end{proof}

%By a similar argument, we can also give a description of irreducible
%$(\GG\ltimes A)$-modules satisfying condition $(J1)$ and
%\begin{enumerate}
%\item[$(J2'')$] The action of $A'$ is nonzero.
%\end{enumerate}
%
%\begin{theorem}\label{char A' nonzero} Let $V$ be an irreducible $(\GG\ltimes
%A)$-module satisfying the conditions $(J1)$ and $(J2'')$, then
%$V\cong F(\a, U, \xi)$ for some irreducible finite dimensional
%$\sl_2$-module $U$, $\a=(\alpha_1,\alpha_2)\in \C^2$ and
%$\xi\in\C\setminus\{0\}$, by adjusting the actions of $t^0$
%appropriately if necessary.
%\end{theorem}
%
%\begin{proof} By Proposition \ref{inj_or_nil}, Lemma \ref{inj} and Lemma \ref{nil}, we can find
%some nonzero $\xi\in\C$ so that $(t^{\bm}t^{\bn}-\xi
%t^{\bm+\bn})V=0$ for all $\bm,\bn,\bm+\bn\neq 0$. Recall that
%$\GG\ltimes A\cong (\GG\ltimes A')\oplus\C t^0$, we can adjust the
%action of $t^0$ as the scalar $\xi$. Then replacing $t^{\bm}$ with
%$(t^{\bm})'=t^{\bm}/\xi$ for all $\bm\in\Z^2$ in the argument of the
%proof of Theorem \ref{char irre jet}, we can deduce that $V\cong
%F(\a,U,\xi)$.
%\end{proof}

\section{Indecomposable jet modules}\label{indecomposable jet modules}

In this section we will give a description of indecomposable jet
$\GG$-modules, which is an analogue of the main result of \cite{B1}.
The arguments in Section \ref{irreducible jet modules} and Section
\ref{LL-modules} depend on the irreducibility of the modules, so the
methods there do not generally apply to indecomposable modules. We
will follow the idea of \cite{B1}, using the technique of polynomial
modules.

From now on, we fix an indecomposable jet $\GG$-modules $V$ with
gradation as in \eqref{grading}.
We first show that $V_0$ is a polynomial
module over $\LL$. Indeed, we have

\begin{theorem}\label{poly_V_0}
Any finite dimensional $\LL$-module $W$ is a polynomial module.
\end{theorem}

\begin{proof} %Since $\LL$ is generated by $L(s, 1), L(s,0), L(s,-1), s\in\Z$,
%it suffices to show that the actions of $L(s, 1), L(s,0)$ and
%$L(s,-1)$ can be represented by polynomials in $s\in\Z$ with
%coefficients in $\End(W)$.
Fix any $s\in\Z\setminus\{0\}$, and we
define the difference derivative inductively by $\partial^0
L(s,1)=L(s,1)$ and $\partial^{l+1}L(s,1)=\partial^l
L(s+1,1)-\partial^l L(s,1)$ for all $l\in\Z_+$. One can check the
following explicit formula by induction on $l$:
$$\partial^lL(s,1)=\sum_{i=0}^l(-1)^{l-i}{l\choose i}L(s+i,1).$$

\noindent\textbf{Claim 1.} $\partial^lL(s,1)\in \mathcal{I}_{k}$ for
all $s\in\Z$ and $l,k\in\Z_+$ with $l> k\geq 2$.

For convenience, we define the linear operators on $\LL$:
$$\phi_{\br}\Big(\sum_{i=1}^na_iL(\bm_i)\Big)=\sum_{i=1}^na_i\bm^{\br}_i,\ \ \forall\ \bm_i\in\Z^2,\br\in\Z_+^2,n\in\N.$$
Then we have $\mathcal{I}_k=\bigcap_{2\leq |\br|\leq
k}\ker\phi_{\br}$ for any $k\geq 2$. Note that
$$\phi_{(j,j')}(\partial^lL(s,1))=\sum_{i=0}^l(-1)^{l-i}{l\choose i}(s+i)^j,\ \forall\ j, j'\in\Z_+.$$
Given any $k\geq 2$, we need only to show
\begin{equation}\label{partial=0}
\sum_{i=0}^l(-1)^{l-i}{l\choose i}(s+i)^j=0,\ \forall\ j, l\in\Z_+
{\rm\ with\ } l>k\geq j.
\end{equation}
The result for $j=0$ is clear. Now suppose that \eqref{partial=0}
holds for any $j\in \Z_+$ with $j\leq k-1$, then we have
\begin{equation*}
\begin{split}
\sum_{i=0}^l(-1)^{l-i}{l\choose i}(s+i)^{j+1} & =\sum_{i=0}^l(-1)^{l-i}i{l\choose i}(s+i)^j\\
 =l\sum_{i=1}^l(-1)^{l-i}{l-1\choose i-1}(s+i)^j & = -l\sum_{i=0}^{l-1}(-1)^{l-i}{l-1\choose i}(s+i+1)^j=0,\emph{}\\
\end{split}
\end{equation*}
where the first and last equalities follow from induction
hypothesis. The claim is proved.

By a similar argument, we can deduce that
$$\p^3(L(r,0))=L(r+3,0)-3L(r+2,0)+3L(r+1,0)-L(r,0)\in\I_2,\ \forall\
r\in\Z.$$ A direct calculation shows that
\begin{equation*}\begin{split}
[\p^lL(s,1),L(r,0)]&=\sum_{i=0}^l(-1)^{l-i}{l\choose i}[L(s+i,1),L(r,0)]\\
&=r\sum_{i=0}^l(-1)^{l-i}{l\choose i}\big(L(s+r+i,1)-L(s+i,1)-L(r,0)\big)\\
&=r\Big(\p^lL(s+r,1)-\p^lL(s,1)\Big).
\end{split}\end{equation*}
Then
\begin{equation*}\begin{split}
 &[\p^lL(s-r,1),\p^3L(r,0)]\\
=&[\p^lL(s-r,1),L(r+3,0)-3L(r+2,0)+3L(r+1,0)-L(r,0)]\\
=&(r+3)\Big(\p^lL(s+3,1)-\p^lL(s-r,1)\Big)-3(r+2)\Big(\p^lL(s+2,1)-\p^lL(s-r,1)\Big)\\
 &+3(r+1)\Big(\p^lL(s+1,1)-\p^lL(s-r,1)\Big)-r\Big(\p^lL(s,1)-\p^lL(s-r,1)\Big)\\
=&(r+3)\p^lL(s+3,1)-3(r+2)\p^lL(s+2,1)+3(r+1)\p^lL(s+1,1)-r\p^lL(s,1)\\
=&r\p^{l+3}L(s,1)+3\p^lL(s+3,1)-6\p^lL(s+2,1)+3\p^lL(s+1,1).
\end{split}\end{equation*}
Replacing $r$ with $r+1$ and making difference, we obtain
$$[\p^lL(s-r-1,1),\p^3L(r+1,0)]-[\p^lL(s-r,1),\p^3L(r,0)]=\p^{l+3}L(s,1).$$
Noticing that $\p^lL(s,1), \p^l(r,0)\in \I_2\cap[\LL,\LL]$ for all
$l\geq 3$, we may deduce by induction that
$$\p^{l+3k}L(s,1)\subseteq (\I_2\cap [\LL,\LL])^{k},\ \forall, k,l\in\N.$$

Now $W$ is a finite dimensional $\LL$-module and hence has a
composition series of submodules. By Theorem \ref{finite}, we see
that $(\I_2\cap [\LL,\LL])^k$ acts trivially on $W$ for sufficiently
large $k\in\N$. In particular, there exists $N\geq 2$ such that
$\p^{l}L(s,1)\in\ker\rho$ for $l>N$, where $\rho$ is the
corresponding representation of $\LL$ in $W$. Then Claim 1 indicates
that $\rho(\partial^lL(s,1))=0$ for all $l\in\Z_+$ with $l>N$. By
Lemma 4.2 in \cite{B1} , there exists a polynomial $f_s$ in one
variable with coefficients in $\End(W)$ such that $\deg(f_s)\leq N$
and $\rho(L(s+i,1))=f_s\emph{}(s+i)$ for all $i\in\Z_+$. Replacing
$s$ with any other $s'\in\Z$, we get the polynomial $f_{s'}$ such
that $\rho(L(s'+i,1))=f_{s'}\emph{}(s'+i)$ for all $i\in\Z_+$. Since
$f_{s}$ and $f_{s'}$ have equal value for infinitely many elements
in $\Z_+$, we have $f_s=f_{s'}$ for all $s,s'\in\Z\emph{}$. Denote
this polynomial by $f$. Then we have $\rho(L(i,1))=f(i)$ in
$\End(W)$ for all $i\in\Z$.

Similarly, we can find polynomial $g$ with coefficients in $\End(W)$
such that $\rho(L(1,i))=g(i)$ for all $i\in\Z$. Then the commutator
\eqref{def_LL} gives that
\begin{equation*}
(ij-1)\rho(L(i+1,j+1))=g(j)f(i)-f(i)g(j)+(ij-1)(g(j)+f(i)),
\end{equation*}
for all $(i,j)\in \Z^2\setminus\{(-1,-1)\}$ or equivalently,
\begin{equation*}
(ij-i-j)\rho(L(i,j))=F(i,j),\ \forall\ (i,j)\in \Z\setminus\{0\},
\end{equation*}
where $F(i,j)=g(j-1)f(i-1)-f(i-1)g(j-1)+(ij-i-j)(g(j-1)+f(i-1))$.
Note that $ij-i-j=0$ implies $F(i,j)=0$, so $F(i,j)$ can be divided
by $ij-i-j$ as polynomials in $i,j$. Denote
$G(i,j)=F(i,j)/(ij-i-j)$, we have
\begin{equation}\label{poly 1}
\rho(L(i,j))=G(i,j),\ \forall\ ij-i-j\neq0.
\end{equation}

Replacing $L(s,1)$ and $L(1,s)$ with $L(s,0)$ and $L(0,s)$
respectively in the previous argument, we may get another polynomial
$H(i,j)$ such that
\begin{equation}\label{poly 2}
\rho(L(i,j))=H(i,j),\ \forall\ ij\neq0.
\end{equation}
Noticing that $G(i,j)=H(j,i)$ for $ij\neq 0$ and $ij-i-j\neq0$, we
have $G(i,j)=H(i,j)$ for all $i,j\in\Z^2$. Combing with the
equations \eqref{poly 1} and \eqref{poly 2}, we conclude that
$\rho(L(i,j))=G(i,j)$, for all $(i,j)\in\Z^2\setminus\{0\}$, as
desired.
\end{proof}

Regard $L\emph{}(\bm), \bm\in\Z^2\setminus\{0\}$ as linear
transformations on $V_0$, and applying Theorem \ref{poly_V_0} to the
$\LL$-module $V_0$ we can find $D_{\bi}\in \End{V_0}, i\in\Z_+^2$,
independent of $\bm$, such that
\begin{equation}\label{D(m)}
L(\bm)=\sum_{\bi\in\Z_+^2}^{\rm{finite}}\frac{\bm^{\bi}}{\bi
!}D_{\bi},\ \forall\ \bm\in\Z^2\setminus\{0\},
\end{equation}
where only finitely many $D_{\bi}\neq 0$. Substitute the above
equation into the commutator of $L(\bm)$ in \eqref{def_LL}, we have
\begin{equation*}%\label{def_GG_+}
\begin{split}
\sum_{\bi\in\Z_+^2}\sum_{\bj\in\Z_+^2}\frac{\bm^{\bi}\bn^{\bj}}{\bi!\bj!}[D_{\bi},D_{\bj}]
= & \det{\bn\choose\bm}\Big( \sum_{\bi\in\Z_+^2}\frac{(\bm+\bn)^{\bi}}{\bi !}D_{\bi}-\frac{\bm^{\bi}}{\bi !}D_{\bi}-\frac{\bn^{\bi}}{\bi !}D_{\bi}\Big)\\
= & \det{\bn\choose\bm}\sum_{\bi, \bj\in\Z_+^2\setminus\{0\}}\frac{\bm^{\bi}\bn^{\bj}}{\bi !\bj !}D_{\bi+\bj}-\det{\bn\choose\bm}D_{(0,0)}\\
= & \sum_{\bi, \bj\in\Z_+^2\setminus\{0\}}\frac{\bm^{\bi+(0,1)}\bn^{\bj+(1,0)}}{\bi !\bj !}D_{\bi+\bj}-\sum_{\bi, \bj\in\Z_+^2\setminus\{0\}}\frac{\bm^{\bi+(1,0)}\bn^{\bj+(0,1)}}{\bi !\bj !}D_{\bi+\bj}\\
  & -\det{\bn\choose\bm}D_{(0,0)}\\
= & \sum_{\bi, \bj\in\Z_+^2, |\bi|\geq 2, |\bj|\geq 2}\Big(\frac{\bm^{\bi}\bn^{\bj}}{\bi !\bj !}\bj(1)\bi(2)D_{\bi+\bj-(1,1)}-\frac{\bm^{\bi}\bn^{\bj}}{\bi !\bj !}\bi(1)\bj(2)D_{\bi+\bj-(1,1)}\Big)\\
&  -\det{\bn\choose\bm}D_{(0,0)}\\
= & \sum_{\bi, \bj\in\Z_+^2, |\bi|\geq 2, |\bj|\geq 2}\frac{\bm^{\bi}\bn^{\bj}}{\bi !\bj !}\det{\bj\choose\bi}D_{\bi+\bj-(1,1)}-\det{\bn\choose\bm}D_{(0,0)}\\
\end{split}
\end{equation*}
Simplifying it and comparing the coefficients of
$\bm^{\bi}\bn^{\bj}/(\bi!\bj!)$, we can obtain
\begin{equation}\label{commotator_D_i}
[D_{\bi}, D_{\bj}] =\left\{
\begin{array}{ll}
\det{\bj\choose\bi}D_{\bi+\bj-(1,1)},&\ \forall\ \bi,\bj\in\Z_+^2\ \rm{with}\ |\bi|\geq 2\ \rm{and}\ |\bj|\geq 2\\\\
D_{(0,0)}\ \rm{or}\ -D_{(0,0)},& \ \rm{if}\ \bi=(1,0),\bj=(0,1)\ \rm{or}\ \bi=(0,1),\bj=(1,0)\\\\
0,& \ \forall\ \rm{other}\ \bi,\bj\in\Z_+^2\\\\
\end{array}\right..
\end{equation}
%or equivalently, dropping the central elements,
%\begin{equation}\label{commotator_D_i+(1,1)}
%\begin{split}
%[D_{\bi+(1,1)}, D_{\bj+(1,1)}]  = \det{\bj+(1,1)\choose\bi+(1,1)}D_{\bi+\bj+(1,1)},\ \forall\ \bi,\bj\in\Z_+^2\ \rm{with}\ |\bi|\geq 2\ \rm{and}\ |\bj|\geq 2.\\
%\end{split}
%\end{equation}
%Denote $b_{\bi}=D_{\bi+(1,1)}$, we have that
%\begin{equation}\label{commotator_b_i}
%\begin{split}
%[b_{\bi}, b_{\bj}] = \det{\bj+(1,1)\choose\bi+(1,1)}b_{\bi+\bj},\ \forall\ \bi,\bj\in\Z_{\geq -1}^2\ \rm{with}\ |\bi|\geq 0\ \rm{and}\ |\bj|\geq 0,\\
%\end{split}
%\end{equation}
%where $\Z_{\geq-1}$ is the set of all integers no less $-1$.

Denote the Lie algebra defined by the above commutator by $\BB_+$
and by abusing language we still use $\{D_{\bi}\,\mid
\bi\in\Z_+^2\}$ as a basis of $\BB_+$. We have
$$\BB_+=\BB'_+\oplus \mathcal{Z},$$
where $\BB'_+=\span\{D_{\bj} \mid \bj\in\Z_+^2, |\bj|\geq
2\}$ and $\mathcal{Z}=\span\{D_{(0,0)}, D_{(0,1)}, D_{(1,0)}\}$. % is the center of $\BB_+$.
Consequently, $V_0$ becomes a finite dimensional $\BB_+$-module.

\begin{remark} Note that the Lie algebra $\BB'_+$ is part of the Lie
algebra $\BB_n=\span\{D_{\bi}\ |\ \bi\in\Z^n\}$ ($n=2$ case) subject
to the following Lie bracket
$$[D_{\bi}, D_{\bj}]=\det{\pi(\bj)\choose\ \pi(\bi)}D_{\bi+\bj-(1,1)},\ \forall\ \bi,\bj\in\Z^n,$$
where $\pi:\Z^n\rightarrow \C^2$ is an additive map.
This algebra is a special case of the Lie algebras of Block type,
first introduced and studied by D. Djokovic and K. Zhao \cite{DZ}.
In a recent paper \cite{IM}, K. Iohara and O. Mathieu have
classified the $\Z^n$-graded simple Lie algebras
$\mathfrak{g}=\oplus_{\bi\in\Z^n}\mathfrak{g}_{\bi}$ with $\dim
\mathfrak{g}_{\bi}=1$ for any $\bi\in\Z^n$, which turns out to be isomorphic to
the algebra $\BB_n$ or the Kac-Moody algebra $A_1^{(2)}$ or
$A_2^{(2)}$. Moreover, K. Iohara \cite{I} studied the graded
$\BB_n$-modules with $1$-dimensional homogeneous spaces.
\end{remark}

The finite dimensional $\BB_+$-module can be characterized by the
following

\begin{lemma}\label{BB-mod}
Any finite codimensional ideal of $\BB'_+$ is of the form
$\mathcal{J}_k=\span\{D(\bj) \mid |\bj|\geq k\}$. Moreover
$[\J_k,\J_l]\subseteq \J_{k+l-2}$ for all $k,l\geq 2$ and
$\BB'_+/\J_3\cong \sl_2$.
\end{lemma}

\begin{proof}
Let $\J$ be an ideal of $\BB'_+$ with $\dim(\BB_+'/\J)<+\infty$. We
first show that $D_{\bi}\in \J$ implies $\J_k\subseteq \J$, where
$k=|\bi|\geq 2$. Indeed, if $D_{\bi}\in\J$, then
\begin{equation*}
\begin{split}
[D_{(1,2)}, D_{\bi}]=(2\bi(1)-\bi(2))D_{\bi+(0,1)}\in\J\\
\end{split}
\end{equation*}
and \begin{equation*}
\begin{split}
[D_{(2,1)}, D_{\bi}]=(\bi(1)-2\bi(2))D_{\bi+(1,0)}\in\J\\
\end{split}
\end{equation*}
imply that $D_{\bi+(0,1)}\in\J$ or $D_{\bi+(1,0)}\in\J$. Inductively
we can find $\bj\in\Z_+^2$ such that $D_{\bj}\in\J$ and $|\bj|=k'$
for any $k'\geq k$. The the following equations
\begin{equation*}
\begin{split}
[D_{(2,0)}, D_{(i,j)}] = \det\left(\begin{array}{cc}i & j\\ 2 & 0\end{array}\right)D_{(i+1,j-1)},\ \forall\ i\geq 0, j\geq 1,\\
\end{split}
\end{equation*}
\begin{equation*}
\begin{split}
[D_{(0,2)}, D_{(i,j)}] = \det\left(\begin{array}{cc}i & j\\ 0 & 2\end{array}\right)D_{(i-1,j+1)},\ \forall\ i\geq 1, j\geq 0,\\
\end{split}
\end{equation*}
implies that $D_{\bj}\in \J$ for all $|\bj|\geq k$. Hence
$\mathcal{J}_k\subseteq \mathcal{J}$.

Then by the commutator \eqref{commotator_D_i}, we have
\begin{equation*}
\begin{split}
[D_{(1,1)}, D_{\bj}]=(\bj(1)-\bj(2))D_{\bj},&\ \forall\ \bj\in\Z_+^2\ \rm{with}\ |\bj|\geq 2.\\
\end{split}
\end{equation*}
We see that $\J$ is the direct sum of eigenspaces with respect to
the action of $D_{(1,1)}$. Take any nonzero eigenvector of
$D_{(1,1)}$, say $x\in\J$ with eigenvalue $l\in\Z$. Without loss of
generality, we may assume that $l\geq 0$ and
$x=\sum_{i=i_0}^na_{i}D_{(i+l,i)}$ for some $a_i\in\C, i_0,
n\in\Z_+$ and $a_n\neq 0$. If $n>i_0$, then a direct calculation
shows that
$$(\ad D_{(2,0)})^n(x)=
\det\left(\begin{array}{cc}n+l & n\\ 2 & 0\end{array}\right)\cdots
\det\left(\begin{array}{cc}2n+l-1 & 1\\ 2 &
0\end{array}\right)a_nD_{(2n+l,0)}\in\J.$$ By the arguments in the
last paragraph, we have $\J_{2n+l}\subset\J$ and in particular
$D_{(n+l,n)}\in\J$. Inductively, we get $D_{(i+l,i)}\in\J$ for all
$i=i_0,\cdots,n$. Thus let $k\in\N$ be the smallest integer such
that $\J_k\cap\J\neq 0$, then $\J=\J_k$.

Finally from the commutator
\begin{equation*}
\begin{split}
[D_{(1,1)}, D_{(2,0)}]=2D_{(2,0)},\ [D_{(1,1)}, D_{(0,2)}]=-2D_{(0,2)}\ \rm{and}\ [D_{(0,2)}, D_{(2,0)}]=4D_{(1,1)}\\
\end{split}
\end{equation*}
we see that $\BB'_+/\J_3\cong\sl_2$.
\end{proof}

Now we can give the description of the indecomposable and
irreducible jet $\GG$-modules using finite dimensional
$\BB_+$-modules.

\begin{theorem} Let $V$ be an indecomposable jet $\GG$-module with
gradation as in \eqref{grading}. Then $V_0$ admits a $\BB_+$-module
structure. Moreover, we have
\begin{enumerate}
\item There exists a one-to-one correspondence between
indecomposable jet $\GG$-modules and indecomposable finite
dimensional $\BB_+$-modules given by
$$V\longmapsto V_0\ \ \rm{and}\ \ U\longmapsto V=\C[t_1^{\pm}, t_2^{\pm}]\otimes U,$$
where the $(\GG\ltimes A)$-action on $\C[t_1^{\pm},
t_2^{\pm}]\otimes U$ is defined as  $\bt^{\bm}\bt^{\bn}\otimes v=\l
\bt^{\bm+\bn}\otimes v$ and
$$d_{\bm}\bt^{\bn}\otimes v=\bt^{\bm+\bn}\otimes\Big(\det{\bn\choose\bm}+\sum_{\bi\in\Z_+^2}\frac{\bm^{\bi}}{\bi!}D_{\bi}\Big)v,\ \forall\ \bm,\bn\in\Z^2, v\in U.$$

\item There exists a one-to-one correspondence between
irreducible jet $\GG$-modules and irreducible finite dimensional
$\sl_2\oplus \C z_1\oplus \C z_2$-modules given by
$$V\longmapsto V_0\ \ \rm{and}\ \ U\longmapsto V=\C[t_1^{\pm}, t_2^{\pm}]\otimes U,$$
where the $(\GG\ltimes A)$-action on $\C[t_1^{\pm},
t_2^{\pm}]\otimes U$ is defined as $\bt^{\bm}\bt^{\bn}\otimes v=\l
\bt^{\bm+\bn}\otimes v$ and
\begin{equation*}\begin{split}
    d_{\bm}\bt^{\bn}\otimes v = \bt^{\bm+\bn}\otimes
    \left(\left(\begin{array}{cc}\bm(1)\bm(2)&-\bm(1)^2\\\bm(2)^2&-\bm(1)\bm(2)\end{array}\right)
    +\det{\textbf{\a}+\bn\choose\bm}\right)v\\
\end{split}\end{equation*}
for any $\bm,\bn\in\Z^2$ and $v\in U$, where  $\a=(\a_1,\a_2)\in \C^2$, and $z_1$ and $z_2$ act on $U$ as scalars $-\a_2$
and $\a_1$ respectively.
\end{enumerate}
\end{theorem}

\begin{proof} %First we have $t^{\bm}t^{\bn}v=\l t^{\bm+\bn}v$ for all $v\in V$.
By the arguments previous Lemma \ref{BB-mod}, we see that any jet
$\GG$-module yields a finite dimensional $\BB_+$-module. On the
other hand, given any finite dimensional $\BB_+$-module $U$, only
finitely many $D_{\bi}, \bi\in\Z_+^2$, have nonzero actions on $U$ by
Lemma \ref{BB-mod}, hence the formula \eqref{D(m)} defines a
representation of $\LL$ on $U$. Setting $V=\bigoplus_{\bn\in\Z^2}
V_{\bn}, V_{\bn}=t^{\bn}\otimes U$, then the formula \eqref{d_mt^n
by L(m)} makes $V$ into a $(\GG\ltimes A)$-module. It is easy to see
that $V$ is an indecomposable (resp. irreducible) $(\GG\ltimes
A)$-module if and only if $U=V_0$ is an indecomposable (resp.
irreducible) $\BB_+$-module.

Now suppose that $U$ is an irreducible $\BB_+$-module, and let $\J$
be the kernel of the corresponding representation. Then $\BB_+/\J$
is reductive. Since $[\J_3,\J_3]\subseteq \J_4$ and
$\BB'_+/\J_3\cong \sl_2$, we see that $\J_4\subseteq \J$ and
$\J_3/\J$ is contained in the radical of $\BB_+/\J$. So all elements
in $\J_3$ act as scalars. On the other hand, we have $\J_3\subseteq
[\BB'_+,\J_3]$, hence $\J_3\subseteq \J$. We conclude that $U$ is a
$\BB'_+/\J_3\cong \sl_2$-module with $D_{(0,0)}, D_{(0,1)}$
and $D_{(1,0)}$ acting as scalars. In particular, $D_{(0,0)}=[D_{(1,0)},D_{(0,1)}]$ acts as $0$.%, say, $c, -\textbf{\a}(2)$ and $\textbf{\a}(1)$.

Choose a standard basis of $\BB'_+/\J_3\cong \sl_2$, say,
$e=-\frac{D_{(2,0)}}{2}, f=\frac{D_{(0,2)}}{2}, h=D_{(1,1)}$. Then
for any $\sl_2$-module $U$, we can regard $U$ as a $\BB_+$-module
via the above isomorphism and actions of $D_{(0,0)}, D_{(0,1)},
D_{(1,0)}$ as scalars $0, \textbf{\a}(1), -\textbf{\a}(2)$
respectively. Forming the $(\GG\ltimes A)$-module $V=A\otimes U$,
the module action can be written as $t^{\bm}(t^{\bn}\otimes v)=\l
t^{\bm+\bn}\otimes v$ for some $\l\neq 0$ and
\begin{equation*}\begin{split}
d_{\bm}t^{\bn}\otimes v & = t^{\bm+\bn}\otimes \left\{ \det{\bn\choose\bm}-\textbf{\a}(2)\bm(1)+\textbf{\a}(1)\bm(2)\right.\\
&\hskip 2cm
+\left.\bm(1)\bm(2)\left(\begin{array}{cc}1&0\\0&-1\end{array}\right)
 -\bm(1)^2\left(\begin{array}{cc}0&1\\0&0\end{array}\right)
 +\bm(2)^2\left(\begin{array}{cc}0&0\\1&0\end{array}\right)\right\}v\\
& = t^{\bm+\bn}\otimes
    \left(\left(\begin{array}{cc}\bm(1)\bm(2)&-\bm(1)^2\\\bm(2)^2&-\bm(1)\bm(2)\end{array}\right)
    +\det{\textbf{\a}+\bn\choose\bm}\right)v\\
\end{split}\end{equation*}
for all $\bm,\bn\in\Z^2$ and $v\in U$.
\end{proof}

\vspace{1mm}

\noindent
X.G.: Department of Mathematics, Zhengzhou university, Zhengzhou 450001,
Henan, P. R. China. Email: guoxq@amss.ac.cn

\vspace{0.2cm} \noindent G.L.: College of Mathematics and Information
Science, Henan University, Kaifeng 475004, China. Email:
liugenqiang@amss.ac.cn

%\vspace{0.2cm} \noindent K.Z.: Department of Mathematics, Wilfrid
%Laurier University, Waterloo, ON, Canada N2L 3C5, and Academy of
%Mathematics and System Sciences, Chinese Academy of Sciences,
%Beijing 100190, P. R. China. Email: kzhao@wlu.ca

\end{document}